\documentclass[11pt,english,american]{amsart}
\usepackage[T1]{fontenc}
\usepackage[dvips,letterpaper]{geometry}
\usepackage[dvips]{xcolor}
\usepackage{amsthm}
\usepackage{amstext}
\usepackage{amssymb}
\PassOptionsToPackage{normalem}{ulem}
\usepackage{ulem}

\makeatletter

\providecommand{\tabularnewline}{\\}
\providecolor{lyxadded}{rgb}{0,0,1}
\providecolor{lyxdeleted}{rgb}{1,0,0}

\numberwithin{equation}{section}
\numberwithin{figure}{section}
\theoremstyle{plain}
\newtheorem{thm}{Theorem}[section]
  \theoremstyle{plain}
  \newtheorem{lem}[thm]{Lemma}
  \theoremstyle{definition}
  \newtheorem{defn}[thm]{Definition}
  \theoremstyle{plain}
  \newtheorem{prop}[thm]{Proposition}
 \theoremstyle{definition}
  \newtheorem{example}[thm]{Example}
  \theoremstyle{plain}
  \newtheorem{cor}[thm]{Corollary}
  \theoremstyle{definition}
  \newtheorem{problem}[thm]{Problem}

\usepackage[ps,dvips,all,cmtip]{xy}
\xyoption{ps}
\usepackage{color}
\usepackage{graphicx}
\usepackage{stmaryrd}

\makeatother

\usepackage{babel}

\begin{document}

\title{Bounds on the Rubbling and Optimal Rubbling Numbers of Graphs}

\author{Gyula Y. Katona}

\author{N\'Andor Sieben}
\begin{abstract}
A pebbling move on a graph removes two pebbles at a vertex and adds
one pebble at an adjacent vertex. Rubbling is a version of pebbling
where an additional move is allowed. In this new move, one pebble
each is removed at vertices $v$ and $w$ adjacent to a vertex $u$,
and an extra pebble is added at vertex $u$. A vertex is reachable
from a pebble distribution if it is possible to move a pebble to that
vertex using rubbling moves. The rubbling number is the smallest number
$m$ needed to guarantee that any vertex is reachable from any pebble
distribution of $m$ pebbles. The optimal rubbling number is the smallest
number $m$ needed to guarantee a pebble distribution of $m$ pebbles
from which any vertex is reachable. We give bounds for rubbling and
optimal rubbling numbers. In particular, we find an upper bound for
the rubbling number of $n$-vertex, diameter $d$ graphs, and estimates
for the maximum rubbling number of diameter $2$ graphs. We also give
a sharp upper bound for the optimal rubbling number, and sharp upper
and lower bounds in terms of the diameter.
\end{abstract}

\address{Budapest University of Technology and Economics Faculty of Electrical
Engineering and Informatics, Department of Computer Science and Information
Theory, H-1521 Budapest Po. box. 91, Hungary }

\email{kiskat@cs.bme.hu}

\address{Northern Arizona University, Department of Mathematics and Statistics,
Flagstaff AZ 86011-5717, USA}

\email{nandor.sieben@nau.edu}

\keywords{pebbling, rubbling}

\subjclass[2000]{05C99}

\maketitle
\def\a{\text{$\shortrightarrow$}}

\date{\the\month/\the\day/\the\year}

\section{Introduction}

Graph pebbling has its origin in number theory. It is a model for
the transportation of resources. Starting with a pebble distribution
on the vertices of a simple connected graph, a \emph{pebbling move}
removes two pebbles from a vertex and adds one pebble at an adjacent
vertex. We can think of the pebbles as fuel containers. Then the loss
of the pebble during a move is the cost of transportation. A vertex
is called \emph{reachable} if a pebble can be moved to that vertex
using pebbling moves. There are several questions we can ask about
pebbling. How many pebbles will guarantee that every vertex is reachable
(\emph{pebbling number}), or that all vertices are reachable at the
same time (\emph{cover pebbling number})? How can we place the smallest
number of pebbles such that every vertex is reachable (\emph{optimal
pebbling number})? For a comprehensive list of references for the
extensive literature see the survey papers \cite{Hurlbert_survey1,Hurlbert_survey2}. 

\emph{Graph rubbling} is an extension of graph pebbling. In this version,
we also allow a move that removes a pebble each from the vertices
$v$ and $w$ that are adjacent to a vertex $u$, and adds a pebble
at vertex $u$. The basic theory of rubbling and optimal rubbling
is developed in \cite{BelSie}. The rubbling number of complete $m$-ary
trees are studied in \cite{Danz}, while the rubbling number of caterpillars
are determined in \cite{Papp}.

The current paper extends the theory of graph rubbling by providing
bounds for the rubbling numbers of graphs. In Section~3, we give
an upper bound for the rubbling number in terms of the number of vertices
and the diameter of the graph. In Sections~4--5, we investigate how
big the rubbling number of diameter 2 graphs can be. Let $f(n,d)$
be the maximum rubbling number of diameter $d$ graphs with $n$ vertices.
We construct a family of graphs whose rubbling numbers match all known
values of $f(n,2)$. We also prove an upper bound for $f(n,2)$. Similar
questions for pebbling are studied in \cite{diamthree,diamtwo}, more
details are given in Section~5. In Section~6, we give a sharp upper
bound for the optimal rubbling number of a graph in terms of the number
of vertices. We also give sharp upper and lower bounds in terms of
the diameter. Similar results for the optimal pebbling number are
presented in \cite{Bunde_optimal,Optimal}. Our results are extensions
of these.

\section{Preliminaries}

Throughout the paper, let $G$ be a simple connected graph. We use
the notation $V(G)$ for the vertex set and $E(G)$ for the edge set.
A \emph{pebble function} on a graph $G$ is a function $p:V(G)\to\mathbb{Z}$
where $p(v)$ is the number of pebbles placed at $v$. A \emph{pebble
distribution} is a nonnegative pebble function. The \emph{size} of
a pebble distribution $p$ is the total number of pebbles $\sum_{v\in V(G)}p(v)$.
We are going to use the notation $p(v_{1},\ldots,v_{n},*)=(a_{1},\ldots,a_{n},q(*))$
to indicate that $p(v_{i})=a_{i}$ for $i\in\{1,\ldots,n\}$ and $p(w)=q(w)$
for all $w\in V(G)\setminus\{v_{1},\ldots,v_{n}\}$.

Consider a pebble function $p$ on the graph $G$. If $\{v,u\}\in E(G)$
then the \emph{pebbling move} $(v,v\a u)$ removes two pebbles at
vertex $v$, and adds one pebble at vertex $u$ to create a new pebble
function\[
p_{(v,v\a u)}(v,u,*)=(p(v)-2,p(u)+1,p(*)).\]
If $\{w,u\}\in E(G)$ and $v\not=w$, then the \emph{strict rubbling
move} $(v,w\a u)$ removes one pebble each at vertices $v$ and $w$,
and adds one pebble at vertex $u$ to create a new pebble function\[
p_{(v,w\a u)}(v,w,u,*)=(p(v)-1,p(w)-1,p(u)+1,p(*)).\]
A \emph{rubbling move is} either a pebbling move or a strict rubbling
move. A \emph{rubbling sequence} is a finite sequence $s=(s_{1},\ldots,s_{k})$
of rubbling moves. The pebble function gotten from the pebble function
$p$ after applying the moves in $s$ is denoted by $p_{s}$. The
pebble function gotten after applying the moves in a multiset $S$
of rubbling moves in any order is denoted by $p_{S}$. The concatenation
of the rubbling sequences $r=(r_{1},\ldots,r_{k})$ and $s=(s_{1},\ldots,s_{l})$
is denoted by $rs=(r_{1},\ldots,r_{k},s_{1},\ldots,s_{l})$.

A rubbling sequence $s$ is \emph{executable} from the pebble distribution
$p$ if $p_{(s_{1},\ldots,s_{i})}$ is nonnegative for all $i$. A
vertex $v$ of $G$ is \emph{reachable} from the pebble distribution
$p$ if there is an executable rubbling sequence $s$ such that $p_{s}(v)\ge1$.
The \emph{rubbling number} $\rho(G)$ of a graph $G$ is the minimum
number $m$ with the property that every vertex of $G$ is reachable
from any pebble distribution of size $m$.

The \emph{optimal rubbling number} $\rho_{\text{opt}}(G)$ of a graph
$G$ is the size of a distribution with the least number of pebbles
from which every vertex is reachable. 

Given a multiset $S$ of rubbling moves on $G$, the \emph{transition
digraph} $T(G,S)$ is a directed multigraph whose vertex set is $V(G)$,
and each move $(v,w\a u)$ in $S$ is represented by two directed
edges $(v,u)$ and $(w,u)$. The transition digraph of a rubbling
sequence $s=(s_{1},\ldots,s_{n})$ is $T(G,s)=T(G,S).$ where $S=\{s_{1},\ldots,s_{n}\}$
is the multiset of moves in $s$. Let $d_{T(G,S)}^{-}$ represent
the in-degree and $d_{T(G,S)}^{+}$ the out-degree in $T(G,S)$. We
simply write $d^{-}$ and $d^{+}$ if the transition digraph is clear
from context.

A multiset $S$ of rubbling moves on $G$ is \emph{balanced} with
a pebble distribution $p$ \emph{at vertex} $v$ if $p_{S}(v)\ge0$.
We say $S$ is \emph{balanced} with $p$ if $S$ is balanced with
$p$ at all $v\in V(G)$, that is, $p_{S}\ge0$. A multiset of rubbling
moves is called \emph{acyclic} if the corresponding transition digraph
has no directed cycles. An element $(v,w\a u)\in S$ is called an
\emph{initial move} of $S$ if $d^{-}(v)=0=d^{-}(w)$ in the transition
digraph. 

An important tool is the following result of \cite{BelSie}.
\begin{lem}
\emph{(No Cycle)} Let $p$ be a pebble distribution on $G$ and $v\in V(G)$.
The following are equivalent.
\begin{enumerate}
\item $v$ is reachable from $p$. 
\item There is a multiset $S$ of rubbling moves such that $S$ is balanced
with $p$ and $p_{S}(v)\ge1$.
\item There is an acyclic multiset $R$ of rubbling moves such that $R$
is balanced with $p$ and $p_{R}(v)\ge1$.
\item Vertex $v$ is reachable from $p$ through an acyclic rubbling sequence.
\end{enumerate}
\end{lem}

\section{Upper bound on the rubbling number}

All the known upper bounds for the pebbling number $\pi$ are also
upper bounds for the rubbling number since $\rho\le\pi$. The following
result is the rubbling version of the upper bound $\pi(G)\le(n-d)(2^{d}-1)+1$
\cite[Theorem 1]{improved}. The difference between the pebbling upper
bound and the rubbling upper bound is $2^{d-1}(n-d-1)\ge0$. The improvement
is $0$ for $P_{n}$ (the path on $n$ vertices) since then $d=n-1$. 
\begin{thm}
If $G$ is a graph with $n$ vertices and diameter $d$, then\[
\rho(G)\le(n-d+1)(2^{d-1}-1)+2.\]
\end{thm}
\begin{proof}
The statement clearly holds if $n=1$ since then $d=0$, so we may
assume that $n\ge2$. Suppose $p$ is a distribution of pebbles from
which vertex $v$ is not reachable. Let $v_{1}$ be a vertex whose
distance is maximal from $v$ and let $Q_{1}$ be the shortest path
between $v$ and $v_{1}$. Recursively define $v_{i+1}$ to be a vertex
in $V(G)\setminus\cup_{j=1}^{i}V(Q_{j})$ whose distance is maximal
from $v$, and define $Q_{i+1}$ to be the shortest path between $v$
and $v_{i+1}$. The recursion must stop after $m\in\mathbb{N}$ steps.
Let $l_{i}$ be the length of $Q_{i}$. Then we have $n>d\ge l_{1}\ge\cdots\ge l_{m}\ge1$
and $m\le n-l_{1}$. If \[
|p|\ge\sum_{i=1}^{m}(2^{l_{i}-1}-1)+2^{l_{1}-1}+1\]
then either some $Q_{i}$ has at least $2^{l_{i}}$ pebbles or there
are some $Q_{j}$ and $Q_{k}$ with at least $2^{l_{j}-1}$ and $2^{l_{k}-1}$
pebbles respectively. In either case $v$ is reachable. So we must
have\begin{eqnarray*}
|p| & < & \sum_{i=1}^{m}(2^{l_{i}-1}-1)+2^{l_{1}-1}+1\le\sum_{i=1}^{m}(2^{l_{1}-1}-1)+2^{l_{1}-1}+1\\
 & = & (m+1)(2^{l_{1}-1}-1)+2\le(n-l_{1}+1)(2^{l_{1}-1}-1)+2\\
 & \le & (n-d+1)(2^{d-1}-1)+2.\end{eqnarray*}
The last inequality follows from the fact that $l_{1}\mapsto(n-l_{1}+1)(2^{l_{1}-1}-1)$
is increasing for $0<l_{1}<n$.
\end{proof}
The upper bound is sharp for $d=0$ since $\rho(K_{1})=1$ and for
$d=1$ since $\rho(K_{n})=2$ for $n>1$. It is also sharp for $d=n-1$
since $\rho(P_{n})=2^{n-1}$. If the diameter of $G$ is 2, then the
upper bound becomes $\rho(G)\le n+1$. This is no surprise since $\rho(G)\le\pi(G)$
and we know \cite{diamtwo} that $\pi(G)$ is either $n$ or $n+1$.
However, this upper bound is not sharp.

\section{Lower bound for $f(n,2)$}

\begin{figure}
\input{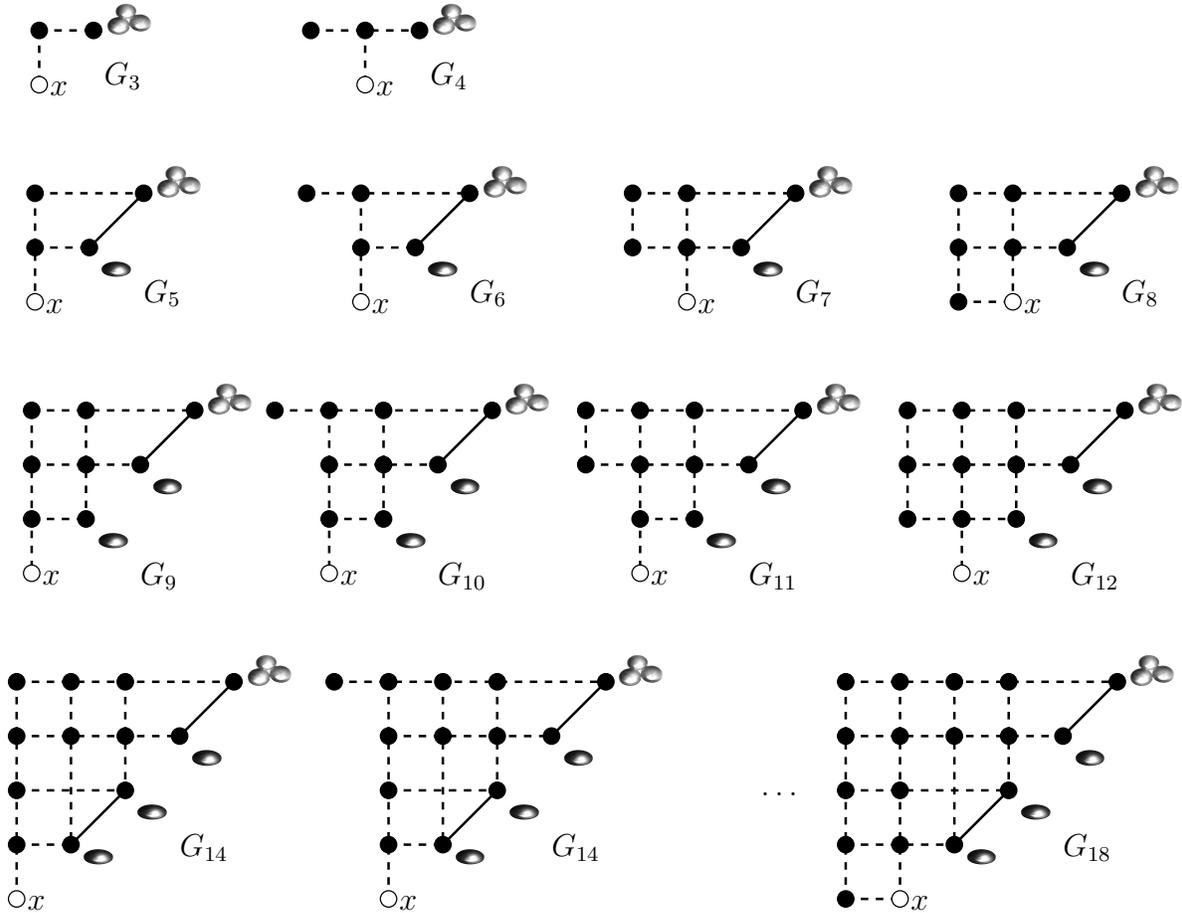}

\caption{\label{fig:Gn}Schematic representation of the graphs $G_{3},\ldots,G_{18}$.
The solid lines are edges of the graphs. The dashed lines indicate
the fact that any two vertices on a horizontal or a vertical line
are connected by an edge. The goal vertex $x$ is not reachable from
the pebble distributions shown on the figures.}

\end{figure}

\begin{figure}
\input{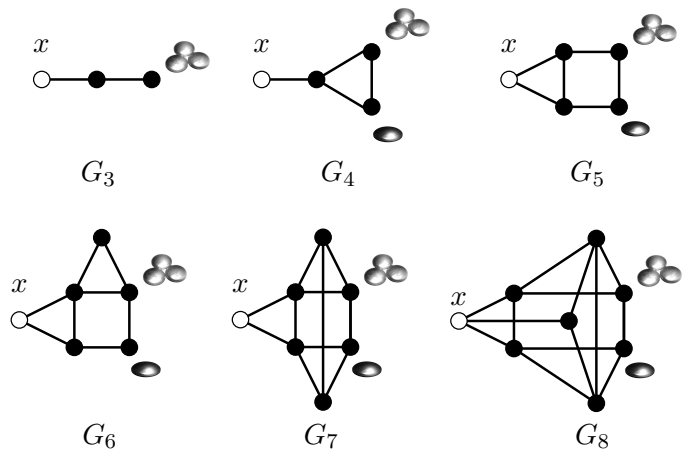}

\caption{\label{fig:Gn2}The graphs $G_{3},\ldots,G_{8}$. The goal vertex
$x$ is not reachable from the pebble distributions shown on the figures.}

\end{figure}

There is no lower bound that forces $\rho$ to grow with the number
of vertices of the graph. In fact $\rho(K_{n})=2$ for all $n$. The
only known lower bound $\rho(G)\ge2^{d}$ for the rubbling number
$\rho$ is coming from the diameter $d$ of the graph $G$. So we
could ask whether $\{\rho(G)\mid\text{diam}(G)=d\}$ is a finite set
for all $d\ge2$. The family of star shaped graphs constructed in
\cite{diamthree} can be used to show that this is not the case for
$d\ge3$. For $d=2$ we need a more elaborate construction.

Our aim is to construct a graph $G_{n}$ for any given $n\ge3$ with
diameter 2 and a high rubbling number. Since the graph is not so easy
to describe directly by giving the vertex and edge sets, we first
define a simpler graph, then we make modifications on it to reach
the final construction.

For a positive integer $s$, let $H_{s}$ be the simple graph defined
by\begin{align*}
V(H_{s}) & =\{(i,j)\mid1\le i\le j\le s\}\\
E(H_{s}) & =\{\{(i_{1},j_{1}),(i_{2},j_{2})\}\mid i_{1}=i_{2}\text{ or }j_{1}=j_{2}\}.\end{align*}

Clearly $|V(H_{s})|=(s+1)s/2$. Now we show that $\text{diam}(H_{n})=2$.
Take two different vertices $(i_{1},j_{1})$ and $(i_{2},j_{2})$
where $i_{1}\leq i_{2}$. If either $i_{1}=i_{2}$ or $j_{1}=j_{2}$,
then they are adjacent. Otherwise $(i_{1},j_{2})\in V(H_{s})$ is
a common neighbor of $(i_{1},j_{1})$ and $(i_{2},j_{2})$, so their
distance is 2.

Now we modify $H_{s}$ by deleting a few vertices and adding a few
more edges in the following way. If $s$ is odd and $s\ge3$, then
delete the vertices\[
(s-1,s),(s-3,s-2),(s-5,s-4),\ldots,(2,3)\]
and add the edges\[
\{(s,s),(s-1,s-1)\},\{(s-2,s-2),(s-3,s-3)\},\ldots,\{(3,3),(2,2)\}.\]
 If $s$ is even, then delete the vertices\[
(s-1,s),(s-3,s-2),(s-5,s-4),\ldots,(3,4)\]
 and add the edges\[
\{(s,s),(s-1,s-1)\},\{(s-2,s-2),(s-3,s-3)\},\ldots,\{(4,4),(3,3)\}.\]

Let $H'_{s}$ denote the graph that is obtained. Clearly $|V(H'_{s})|=(s+1)s/2-\lfloor(s-1)/2\rfloor$.
One can see that $\text{diam}(H'_{s})=2$ holds too, since any pair
of vertices whose unique common neighbor was deleted, now is either
connected by an edge or has a new common neighbor on the spine. 
\begin{defn}
If for a given $n\ge3$ we have $|V(H'_{s})|=n$ for some $s$, then
let $G_{n}=H'_{s}$. Thus, we have the construction of $G_{n}$ for
$n=3,5,9,13,\dots$. For the values of $n$ where $|V(H'_{s})|<n<|V(H'_{s+1})|$,
the construction is given by adding some vertices and edges to $H'_{s}$.
We add the vertices $(0,s),(0,s-1),(0,s-2),\ldots,(0,1)$ one by one
until we reach the required $n$ vertices. A new vertex $(0,j)$ is
adjacent to another vertex $(i',j')$ if either $i'=0$ or $j=j'$. 
\end{defn}
A visualization of the graph family $G_{n}$ is shown in Figures~\ref{fig:Gn}
and \ref{fig:Gn2}. Roughly speaking, we add the new vertices on the
left of the graph, starting at the top row and continuing towards
the bottom. We create new edges to keep the general edge structure
of the graph. We stop adding new vertices before we reach the number
of vertices in $H'_{s+1}$. This means that we stop at vertex $(0,1)$
if $s$ is odd and stop at $(0,2)$ if $s$ is even. Graphs $G_{8}$
and $G_{12}$ shown in Figure~\ref{fig:Gn} illustrate these differently
placed last new vertices. Note that $G_{9}=H'_{4}$ and $G_{13}=H'_{5}$. 

Note that\begin{align*}
|V(H'_{s})|+s & =(s+1)s/2-\lfloor(s-1)/2\rfloor+s\\
 & =(s+2)(s+1)/2-\lfloor s/2\rfloor-1=|V(H'_{s+1})|-1\end{align*}
if $s$ is odd, and $|V(H'_{s})|+s-1=|V(H'_{s+1})|-1$ if $s$ is
even. 

The $i$-th \emph{row} of $G_{n}$ is $R_{i}=\{(i',j')\in V(G_{n})\mid i'=i\}$
while the $i$-th \emph{column} of $G_{n}$ is $C_{i}=\{(i',j')\in V(G_{n})\mid j'=j\}$.
The \emph{spine} of $G_{n}$ is $\{(i,i)\in V(G_{n})\}$. A short
calculation shows that $G_{n}$ has $\left\lfloor \sqrt{2n-1}\right\rfloor $
rows. It is easy to see that $\text{diam}(G_{n})=2$. 
\begin{prop}
For $n\ge3$ we have $\rho(G_{n})\le\left\lfloor \sqrt{2n-1}\right\rfloor +2$.\end{prop}
\begin{proof}
Let $k=\left\lfloor \sqrt{2n-1}\right\rfloor $ be the number of rows
in $G_{n}$. Let $p$ be a pebble distribution on $G_{n}$ containing
$k+2$ pebbles and suppose that a goal vertex $x=(i_{0},j_{0})$ is
not reachable. We are going to define some collections of rows and
columns, but in one case it is just part of the columns. So let us
first define the partial columns: $C'_{i}=\left\{ (i',j')\mid i'=i\mbox{ and }j'<j_{0}\right\} .$
Now let \begin{gather*}
\mathcal{R}=\{R_{j}\mid j>j_{0}\},\quad R=\cup\mathcal{R},\\
\mathcal{C}=\{C'_{i}\mid i\not=i_{0}\},\quad C=\cup\mathcal{C},\\
X=R_{j_{0}}\cup C_{i_{0}},\end{gather*}
as shown in Figure~\ref{fig:RCX}. Then $|\mathcal{R}|=k-j_{0}$,
$j_{0}-2\leq|\mathcal{C}|\le j_{0}-1$ (note that $C'_{0}$ may be
empty) and $R\cup C\cup X=V(G_{n})$.%
\begin{figure}
\input{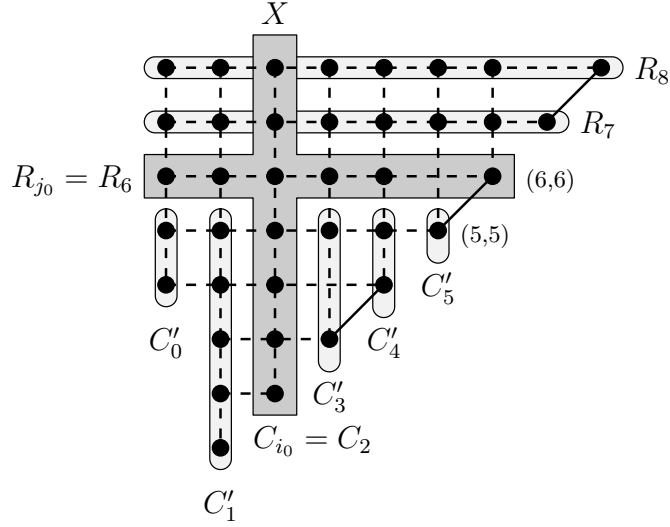}

\caption{\label{fig:RCX}The sets $\mathcal{R},\mathcal{C}$ and $X$..}

\end{figure}

If $R_{j}\in\mathcal{R}$ and there are two or three pebbles on $R_{j}$,
then we can apply a rubbling move on those pebbles and move a pebble
to $C_{i_{0}}$. Row \foreignlanguage{english}{$R_{j}$} cannot have
$4$ or more pebbles, because then we can move two pebbles to $X,$
so $x$ would be reachable. Similarly, if $C'_{i}\in\mathcal{C}$
and $C'_{i}$ has two or three pebbles, then we can move a pebble
to $R_{j_{0}}$, while if it has $4$ or more pebbles, then we can
move two pebbles to $X$. In the exceptional case, when $R_{j_{0}}$
has no vertex in the $i$-th column (like the 5th column on Figure~\ref{fig:RCX}),
we can move the pebbles along the spine to $(i+1,j_{0})$ instead
of $(i,j_{0})$. 

If we can move two pebbles to $X$, then $x$ is clearly reachable.
That means that if $R$ or $C$ has one or two more pebbles than the
size of $\mathcal{R}$ or $\mathcal{C}$ respectively, then we can
move one pebble to $X$. If it has three more pebbles, then we can
move two pebbles to $X.$ So if either $R$ or $C$ contains enough
pebbles to move two pebbles to $X$ or both of them contains enough
pebbles to move one pebble to $X$, then $x$ is reachable.

We know that $p$ can have at most one pebble on $X$. If $X$ has
no pebbles, then the total number of pebbles cannot be more than\[
|\mathcal{R}|+|\mathcal{C}|+2\leq(k-j_{0})+(j_{0}-1)+2=k+1\]
 which is not possible. If $X$ has exactly one pebble, then the total
number of pebbles cannot be more than \[
|\mathcal{R}|+|\mathcal{C}|+1\leq(k-j_{0})+(j_{0}-1)+1=k\]
which is not possible either.
\end{proof}
In the proof of our next result, we are going to need to keep track
of the movement of pebbles during rubbling moves. For this purpose,
we need to replace our pebbles with dependency sets. 
\begin{defn}
A \emph{dependency distribution} is a partition $\mathcal{P}$ of
an \emph{initial pebble set} together with a \emph{location function}
$l:\mathcal{P}\to V(G)$. The elements of $\mathcal{P}$ are called
\emph{dependency sets} or simply pebbles.
\end{defn}
We think of a dependency set as a pebble with some additional information
about the history of the pebble. Dependency distributions replace
pebble distributions. Given a pebble distribution $p$ containing
$m$ pebbles, we can create a \emph{corresponding dependency distribution}
$\mathcal{P}=\{\{1\},\ldots,\{m\}\}$ such that $|\{A\in\mathcal{P}\mid l(A)=v\}|=p(v)$
for all $v\in V(G)$.
\begin{defn}
If $l(A)$ and $l(B)$ are both adjacent to $u$, then the \emph{rubbling
move} $(A,B\a u)$ removes the dependency sets $A$ and $B$ from
$\mathcal{P}$ at the vertices $l(A)$ and $l(B)$ respectively, and
adds a new dependency set $A\cup B$ to $\mathcal{P}$ with location
$l(A\cup B)=u$. A vertex $v$ is \emph{reachable} from a dependency
distribution if a dependency set $A$ with $l(A)=v$ can be created
using rubbling moves.
\end{defn}
Note that the rubbling move $(A,B\a u)$ is essentially the rubbling
move $(l(A),l(B)\a u)$ with some additional information about the
history of the pebbles. It is clear that a vertex is reachable from
a pebble distribution if and only if it is reachable from the corresponding
dependency distribution. Also note that in a rubbling move $(A,B\a u)$
we must have $A\cap B=\emptyset$. 
\begin{defn}
Let $(s_{1},\ldots,s_{k})$ be a sequence of rubbling moves in a dependency
distribution. We say $s_{i}=(A,B\a u)$ is \emph{dependent} on $s_{j}=(C,D\a w)$
if $C\cup D\subseteq A\cup B$. We say that $s_{i}$ and $s_{j}$
are \emph{independent} if $A\cup B$ and $C\cup D$ are disjoint.
\end{defn}
Roughly speaking, $s_{i}$ and $s_{j}$ are independent if they rely
on two disjoint sets of pebbles. It is clear that dependence of rubbling
moves is a transitive relation. Also note that $s_{i}$ and $s_{j}$
are independent if and only if neither $s_{i}$ is dependent on $s_{j}$
nor $s_{j}$ is dependent on $s_{i}$. 
\begin{example}
Consider the initial dependency distribution $l(\{1\},\{2\},\{3\},\{4\},\{5\})=(w,w,w,w,v)$
and the sequence of rubbling moves \[
s_{1}=(\{1\},\{2\}\a u),s_{2}=(\{3\},\{4\}\a v),s_{3}=(\{3,4\},\{5\}\a u),s_{4}=(\{1,2\},\{3,4,5\}\a x).\]
Then $s_{1}$ and $s_{3}$ are independent but $s_{4}$ depends on
$s_{1}$.\end{example}
\begin{prop}
For $n\ge3$ we have $\rho(G_{n})>\left\lfloor \sqrt{2n-1}\right\rfloor +1$.\end{prop}
\begin{proof}
Let $k=\left\lfloor \sqrt{2n-1}\right\rfloor $ be the largest column
index in $G_{n}$. We show that the goal vertex $x:=(1,1)$ is not
reachable from the pebble distribution that has a single pebble on
vertex $(i,i)$ for $1<i<k$ and three pebbles on vertex $(k,k)$
as shown in Figures~\ref{fig:Gn} and \ref{fig:Gn2}. To see this,
we show that $x$ is not reachable from the dependency distribution
$\mathcal{P}=\{\{1\},\ldots,\{k+1\}\}$ with \[
l(\{1\},\ldots,\{k-2\},\{k-1\},\{k\},\{k+1\})=((2,2),\ldots,(k-1,k-1),(k,k),(k,k),(k,k)).\]
For a contradiction suppose that $x$ is reachable from $\mathcal{P}$,
that is, there is a sequence $s_{1},\ldots,s_{m}$ of rubbling moves
that creates a dependency set at $x$.

Let us call a rubbling move $(A,B\a u)$ horizontal if $l(A)$, $l(B)$
and $u$ are all contained in the same row of $G_{n}$. To reach the
goal vertex $x$, we must use a rubbling move involving vertices on
$C_{1}\cup C_{0}$. There are no pebbles on these vertices originally
and the only way to move a new pebble there is to use a horizontal
rubbling move. So there must be at least two independent horizontal
moves $s_{i}=(A,B\a u)$ and $s_{j}=(C,D\a w)$ in our rubbling sequence.
We show that this is not possible.

Since $s_{i}$ and $s_{j}$ are independent, at least one of the sets
$A\cup B$ and $C\cup D$ contains at most one element of $\{k-1,k,k+1\}$.
Roughly speaking, this means that both $s_{i}$ and $s_{j}$ cannot
rely on more than one pebble available at vertex $(k,k)$ in $\mathcal{P}$
since there are only three pebbles there. So we can assume that $(A\cup B)\cap\{k,k+1\}=\emptyset$. 

Now we create a new dependency distribution $\tilde{\mathcal{P}}=\mathcal{P}\setminus\{\{k\},\{k+1\}\}$
by removing the two pebbles from $\mathcal{P}$ that $s_{i}$ does
not rely on for sure. We also remove the rubbling moves from $(s_{1},\ldots,s_{m})$
that are dependent on $k$ or $k+1$. More precisely, we remove the
rubbling moves of the form $(K,L\a v)$ for which $(K\cup L)\cap\{k,k+1\}\not=\emptyset$.
The resulting rubbling sequence $(\tilde{s}_{1},\ldots,\tilde{s}_{\tilde{m}})$
is executable from $\tilde{\mathcal{P}}$ and contains $s_{i}$.

Let us call $L_{i}=R_{i}\cup C_{i}$ for $i\ge1$ a \emph{line} of
$G_{n}$. Line $L_{i}$ contains the spine vertex $(i,i)$. Note that
all vertices on the spine are contained in exactly one line, and all
other vertices are contained in at most two lines. We say that a pebble
configuration is \emph{forbidden} if there is a line with more than
one pebble. Note that two pebbles on a row $R_{i}$ or on a column
$C_{i}$ with $i\ge1$ is a forbidden configuration. It is clear that
$\tilde{\mathcal{P}}$ is not a forbidden pebble configurations. 

\begin{figure}
\input{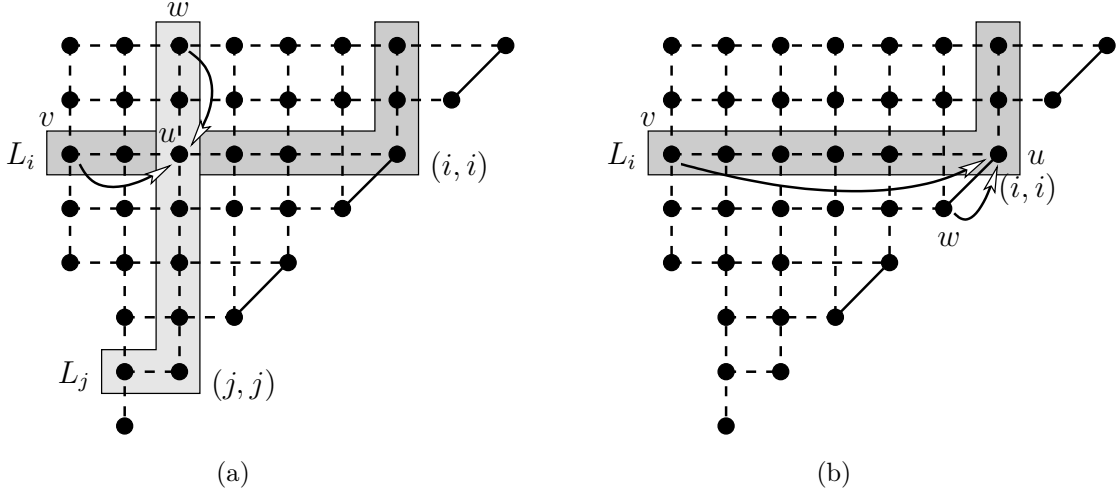}

\caption{\label{fig:linei-1}Possible ways to move a pebble to the line $L_{i}$.}

\end{figure}

We are going to show that a rubbling move cannot create a forbidden
configuration if there was no forbidden configuration before the rubbling
move. Suppose that a rubbling move $(A,B\a u)$ creates a pebble on
vertex $u$ of line $L_{i}$. If $u$ is not on the spine, then it
is contained in another line $L_{j}$, so all of its neighbors are
in $L_{i}\cup L_{j}$. Both $l(A)$ and $l(B)$ cannot be on the same
line since there was no forbidden configuration before this step.
Thus one of $l(A)$ and $l(B)$ must be on $L_{i}$ while the other
must be outside of $L_{i}$ as shown on Figure~\ref{fig:linei-1}(a).
If $u$ is on the spine, then it has a neighbor which is on the spine,
too. Again, both of $l(A)$ and $l(B)$ cannot be on this vertex,
because that is a forbidden configuration. Therefore one of $l(A)$
and $l(B)$ must be on $L_{i}$ again as shown on Figure~\ref{fig:linei-1}(b)
since in this case all other neighbors of $u$ are in $L_{i}$. 

Thus we can assume that $l(A)=v$ is in $L_{i}$. Since $v$ is in
$L_{i}$, we cannot have any other pebble on $L_{i}$ before the rubbling
move. So $u$ is the only pebble on $L_{i}$ after the rubbling move
and so the rubbling move did not create any forbidden configurations.

We saw that $(\tilde{s}_{1},\ldots,\tilde{s}_{\tilde{m}})$ has the
horizontal move $s_{i}$. A horizontal move is only possible if there
are two pebbles on a row which is a forbidden configuration. This
is a contradiction since we do not have any forbidden configurations
during the execution of $(\tilde{s}_{1},\ldots,\tilde{s}_{\tilde{m}})$.\end{proof}
\begin{cor}
For $n\ge3$ we have $\rho(G_{n})=\left\lfloor \sqrt{2n-1}\right\rfloor +2$.
\end{cor}

\section{Upper bound for $f(n,2)$}

\begin{table}
\begin{tabular}{|c|c|c|c|c|c|c|c|c|}
\hline 
$n$ & 3 & 4 & 5 & 6 & 7 & 8 & 9 & 10\tabularnewline
\hline 
$f(n,2)$ & $4$ & $4$ & $5$ & $5$ & $5$ & $5$ & $6$ & ?\tabularnewline
\hline 
$\rho(G_{n})$ & $4$ & $4$ & $5$ & $5$ & $5$ & $5$ & $6$ & $6$\tabularnewline
\hline
\end{tabular}

~

~

\caption{\label{cap:Maximum-rubbling-numbers}Rubbling numbers of $G_{n}$
and all known maximum rubbling numbers for diameter 2 graphs with
$n$ vertices. }

\end{table}

Table~\ref{cap:Maximum-rubbling-numbers} shows the maximum rubbling
numbers \[
f(n,2)=\max\{\rho(G)\mid n=|V(G)|\text{ and }2=\text{diam}(G)\}\]
of diameter 2 graphs with $n$ vertices. The values were calculated
by a computer program \cite{pebbleAlgo}. The program checked all
diameter 2 graphs with a given number of vertices. These graphs were
generated by Nauty \cite{nauty}. We have $f(n,2)=\rho(G_{n})$ for
$n\in\{3,\ldots,9\}$. It is not clear whether this is true for all
$n$.
\begin{problem}
Is it true that $f(n,2)=\rho(G_{n})$ for all $n\ge3$?
\end{problem}
There are more existing results for similar questions about pebbling.
It is known \cite{Pachter} that $f(n,2)=n+1$ since the pebbling
number of a diameter 2 graph is either $n$ or $n+1$. A classification
of diameter 2 graphs with pebbling number $n+1$ is also known from
\cite{diamtwo}. Diameter 3 graphs are also studied. In \cite{diamthree},
it is shown that $f(n,3)=\frac{3}{2}n+O(1)$. 

The proof of the following result uses the method of \cite{Hall}.
\begin{lem}
\label{lem:hyper}Let $\mathcal{H}$ be a 3-uniform hypergraph on
$q$ vertices. If $|E\cap F|\not=1$ for all $E,F\in\mathcal{H}$,
then $|\mathcal{H}|\le q$.\end{lem}
\begin{proof}
Let $\mathbf{v}_{1},\ldots,\mathbf{v}_{n}$ denote the characteristic
vectors of the sets in $\mathcal{H}$. We claim that the characteristic
vectors are linearly independent over $GF(2)$. This clearly implies
the result. 

Since every set contains $3$ elements, we have $\mathbf{v}_{i}^{2}=1$
for all $i$ since $3\equiv_{2}1$. On the other hand, if $i\not=j$
then $\mathbf{v}_{i}\cdot\mathbf{v}_{j}=0$ since the product is the
cardinality of the intersection of the two corresponding sets. If
$\sum_{i}^{n}c_{i}\mathbf{v}_{i}=\mathbf{0}$ then multiplying by
$\mathbf{v}_{j}$ we obtain $c_{j}=c_{j}\mathbf{v}_{j}^{2}=\mathbf{0}\cdot\mathbf{v}_{j}=0$.
This proves our claim.
\end{proof}
The set of vertices adjacent to a given vertex $v$ of a graph is
denoted by $N(v)$.
\begin{prop}
\label{pro:nofive}Let $G$ be a diameter two graph with goal vertex
$x$. Let $p$ be a pebble distribution on $G$ containing $m$ pebbles.
If $\{x\}\cup N(x)$ is not reachable using only five pebbles of $p$,
then $G$ has at least $\left\lfloor \frac{m^{2}+3}{2}\right\rfloor $
vertices.
\end{prop}
\begin{figure}
~\input{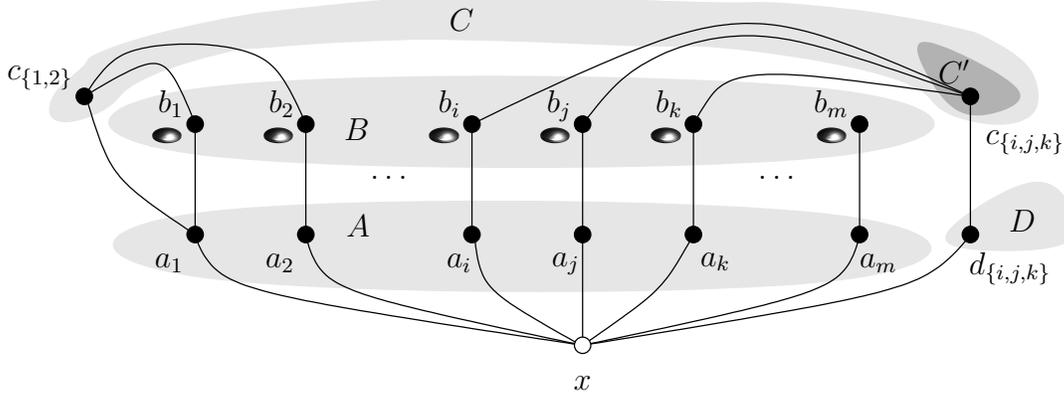}

\caption{\label{cap:nofive}The sets $A,B,C,C'$ and $D$ in the proof of Proposition
\ref{pro:nofive}.}

\end{figure}

\begin{proof}
For $v\in N(x)$ let $R(v)=N(v)\cup\{v\}\setminus\{x\}$. Since the
diameter of $G$ is two, we must have\[
V(G)\setminus\{x\}=\bigcup_{v\in N(x)}R(v).\]
Since $\{x\}\cup N(x)$ is not reachable using only five pebbles,
no $R(v)$ can contain more than two pebbles otherwise we could move
a pebble to $v\in N(x)$. So there is a maximal subset $A=\{a_{1},\ldots,a_{m}\}$
of $N(x)$ such that $R(a_{i})$ contains exactly one pebble on a
vertex $b_{i}\in R(a_{i})$ for all $i$. Define $B=\{b_{1},\ldots,b_{m}\}$
and note that $A$ and $B$ are clearly disjoint. To simplify notation
we write $R_{i}$ for $R(a_{i})$. Figure~\ref{cap:nofive} shows
a visualization of these sets.

If two vertices $b_{i}$ and $b_{j}$ of $B$ are not adjacent, then
the diameter condition implies that we can pick a common neighbor
$c_{\{i,j\}}$ of $b_{i}$ and $b_{j}$. Note that $c_{\{i,j\}}=c_{\{j,i\}}$.
If $i$, $j$, $k$ and $l$ are all different, then $c_{\{i,j\}}\not=c_{\{k,l\}}$,
otherwise $(b_{i},b_{j}\a c_{\left\{ i,j\right\} })(b_{k},b_{l}\a c_{\left\{ i,j\right\} })$
would move two pebbles to $c_{\{i,j\}}$ from which $N(x)$ is reachable
using only four pebbles. 

If $c_{\{i,j\}}=c_{\{i,k\}}$ then we write $c_{\{i,j,k\}}$ for $c_{\{i,j\}}=c_{\{i,k\}}$.
Define\[
\mathcal{J}=\{\{i,j\}\mid\{b_{i},b_{j}\}\in E(\overline{G})\},\quad C=\{c_{\{i,j\}}\mid\{i,j\}\in\mathcal{J}\}.\]
A vertex $b_{i}$ cannot be adjacent to two different vertices $b_{j}$
and $b_{k}$ of $B$, otherwise $(b_{j},b_{k}\a b_{i})$ would create
two pebbles on $b_{i}$, making $a_{i}\in N(x)$ reachable using only
three pebbles. Hence the number of edges between the elements of $B$
is at most $\left\lfloor \frac{m}{2}\right\rfloor $ and so $|\mathcal{J}|\ge{m \choose 2}-\left\lfloor \frac{m}{2}\right\rfloor $.

We also introduce a subset $C'$ of $C$ by letting \[
\mathcal{K}=\{\{i,j,k\}\mid c_{\{i,j\}}=c_{\{i,k\}}\},\quad C'=\{c_{\{i,j,k\}}|\{i,j,k\}\in\mathcal{K}\}.\]
Though it may happen that $c_{\{i,j\}}=c_{\{i,k\}}=c_{\{j,k\}}$,
we have $|C\setminus C'|\ge\mathcal{J}-3|C'|$. Note that $c_{\{i,j\}}\in C\setminus C'$
and $\{i,j\}\not=\{k,l\}$ implies $c_{\{i,j\}}\not=c_{\{k,l\}}$.
If $c_{\{i,j,k\}}=c_{\{i',j',k'\}}$ then $\{i,j,k\}=\{i',j',k'\}$
otherwise we could move two pebbles to $c_{\{i,j,k\}}$ using only
four pebbles. Hence $|\mathcal{K}|=|C'|$.

We can clearly move a pebble to any element of $C$ and so $C\cap A\subseteq C\cap N(x)=\emptyset$.
If $k\not\in\{i,j\}$ then $c_{\{i,j\}}\not\in R_{k}$ otherwise $(b_{i},b_{j}\a c_{\{i,j\}})(b_{k},c_{\{i,j\}}\a a_{k})$
would move a pebble to $a_{k}\in N(x)$ using only three pebbles.
In particular, $c_{\{i,j\}}\not=b_{k}$ and so $C\cap B=\emptyset$. 

Suppose that $c_{\{i,j,k\}}\in C'$. Then $c_{\{i,j,k\}}\not\in R_{l}$
for $l\not\in\{i,j,k\}$ since $c_{\{i,j,k\}}$ is in $\{c_{\{i,j\}},c_{\{i,k\}},c_{\{j,k\}}\}$.
We also have $c_{\{i,j,k\}}\not\in R_{i}$, otherwise $(b_{j},b_{k}\a c_{\{i,j,k\}})(b_{i},c_{\{i,j,k\}}\a a_{i})$
would move a pebble to $a_{i}\in N(x)$ using only three pebbles.
Similar arguments show that $c_{\{i,j,k\}}\not\in R_{j}\cup R_{k}$.
Hence\[
c_{\{i,j,k\}}\in R(d_{\{i,j,k\}})\setminus\bigcup_{l=1}^{m}R_{l}\quad\text{for some}\quad d_{\{i,j,k\}}\in N(x)\setminus A.\]

Let $D=\{d_{\{i,j,k\}}\mid\{i,j,k\}\in\mathcal{K}\}$. Then $D$ is
disjoint from $A\cup B$ by definition. We also have $D\cap C\subseteq N(x)\cap C=\emptyset$.
If $d_{\{i,j,k\}}=d_{\{i',j',k'\}}$ then $\{i,j,k\}=\{i',j',k'\}$,
otherwise $\{b_{i},b_{j},b_{k},b_{i'},b_{j'},b_{k'}\}$ would have
at least four vertices with pebbles so we could move a pebble to $d_{\{i,j,k\}}\in N(x)$
using only these four pebbles. So $D$ and $C'$ has the same number
of elements.

The intersection of two elements $\{i,j,k\}$ and $\{i,j',k'\}$ of
$\mathcal{K}$ cannot be a singleton set$\left\{ i\right\} $, otherwise
\[
(b_{j},b_{k}\a c_{\{j,k\}})(b_{j'},b_{k'}\a c_{\{j',k'\}})(c_{\{j,k\}},c_{\{j',k'\}}\a b_{i})\]
 would move two pebbles to $b_{i}$ from where $\{x\}\cup N(x)$ would
be reachable using only five pebbles. Hence the 3-uniform hypergraph
$\mathcal{K}$ satisfies the conditions of Lemma~\ref{lem:hyper}
and so $|\mathcal{K}|\le m$.

The result now follows from the calculation\begin{align*}
|V_{G}| & \ge|\{x\}|+|A|+|B|+|C|+|D|=1+m+m+|C\setminus C'|+|C'|+|D|\\
 & \ge1+2m+|\mathcal{J}|-3|C'|+|C'|+|C'|\ge1+2m+{m \choose 2}-\left\lfloor \frac{m}{2}\right\rfloor -|C'|\\
 & =\frac{2+3m+m^{2}}{2}-\left\lfloor \frac{m}{2}\right\rfloor -|\mathcal{K}|\ge\frac{2+3m+m^{2}}{2}-\left\lfloor \frac{m}{2}\right\rfloor -m=\left\lfloor \frac{m^{2}+3}{2}\right\rfloor .\end{align*}
\end{proof}
\begin{prop}
\label{pro:mton}Let $G$ be a diameter two graph with goal vertex
$x$. Let $p$ be a pebble distribution on $G$ containing $m\ge5$
pebbles. If $x$ is not reachable from $p$ then $G$ has at least
$\left\lfloor \frac{(m-5)^{2}+3}{2}\right\rfloor $ vertices.\end{prop}
\begin{proof}
If $p$ satisfies the conditions of Proposition~\ref{pro:nofive},
then $G$ must have at least $\left\lfloor \frac{m^{2}+3}{2}\right\rfloor $
vertices. 

Otherwise there are five pebbles of $p$ from where $\{x\}\cup N(x)$
is reachable. We can remove these pebbles to create a new pebble distribution
$q$ containing $m-5$ pebbles. Then $q$ must satisfy the conditions
of Proposition~\ref{pro:nofive}, otherwise we could move two pebbles
to $\{x\}\cup N(x)$ and so $x$ would be reachable. So $G$ must
have at least $\left\lfloor \frac{(m-5)^{2}+3}{2}\right\rfloor $
vertices.\end{proof}
\begin{prop}
The rubbling number of a diameter 2 graph with $n$ vertices cannot
be larger than $\sqrt{2n-1}+5$.\end{prop}
\begin{proof}
From Proposition~\ref{pro:mton}, we know that if we have $m\ge5$
pebbles on a diameter 2 graph with $n$ vertices one of which is not
reachable, then $n\ge\left\lfloor \frac{(m-5)^{2}+3}{2}\right\rfloor $.
The contrapositive gives that if a graph has $n<\left\lfloor \frac{(m-5)^{2}+3}{2}\right\rfloor $
vertices and $m\ge5$, then any goal vertex is reachable from any
placement of $m$ pebbles and so the rubbling number is larger than
$m$. The result now follows since we have\begin{align*}
n<\left\lfloor \frac{(m-5)^{2}+3}{2}\right\rfloor  & \iff n+1\le\frac{(m-5)^{2}+3}{2}\\
 & \iff\sqrt{2n-1}+5\le m.\end{align*}
\end{proof}
\begin{cor}
We have $\lfloor\sqrt{2n-1}\rfloor+2\le f(n,2)\le\sqrt{2n-1}+5$.
\end{cor}

\section{Bounds on the optimal rubbling number}

We saw in \cite{BelSie} that $\rho_{\text{opt}}(P_{n})=\left\lceil \frac{n+1}{2}\right\rceil $.
We show that the path requires the most pebbles for optimal rubbling
amongst the graphs with a given number of vertices. The proof follows
the ideas of \cite{Bunde_optimal}. 
\begin{prop}
\label{pro:opt}If $G$ is a tree with $n$ vertices, then $\rho_{\text{opt}}(G)\le\left\lceil \frac{n+1}{2}\right\rceil $. \end{prop}
\begin{proof}
We use induction on $n$. The statement is clearly true for $n\in\{1,2\}$.
In the inductive step let $n\ge3$ and let $v_{1},v_{2},\ldots,v_{k}$
be the consecutive vertices of a longest path of $G$. Note that $k\ge3$.
We are going to find a subtree $H$ of $G$ with $n-2$ vertices as
shown in Figure~\ref{fig:opt}. Then there is a pebble distribution
$q$ on $H$ with size $\left\lceil \frac{n-2+1}{2}\right\rceil $
from which every vertex of $H$ is reachable. 

If $d(v_{2})=2$ then let $H$ be the subtree of $G$ gotten by deleting
$v_{1}$ and $v_{2}$. Then every vertex of $G$ is reachable from
the pebble distribution $p(v_{1},v_{2},*):=(1,0,q(*))$.

If $d(v_{2})>2$ then let $w$ be a vertex that is adjacent to $v_{2}$
but different from $v_{1}$ and $v_{3}$. By the maximality of the
chosen path, $w$ must be a leaf vertex. Let $H$ be the subtree of
$G$ gotten by deleting $v_{1}$ and $w$. Then every vertex of $G$
is reachable from the pebble distribution $p(v_{1},v_{2},w,*):=(0,q(v_{2})+1,0,q(*))$.

The size of $p$ is $\left\lceil \frac{n-1}{2}\right\rceil +1=\left\lceil \frac{n+1}{2}\right\rceil $
in both cases as desired.
\end{proof}
\begin{figure}
\input{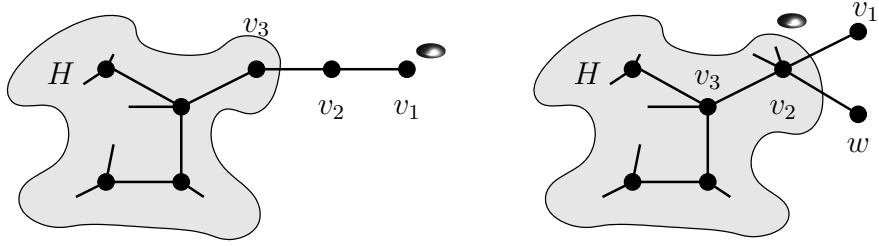}

\caption{\label{fig:opt}The construction of $H$ in the proof of Proposition~\ref{pro:opt}.}

\end{figure}

\begin{cor}
If $G$ is a connected graph with $n$ vertices, then $\rho_{\text{opt}}(G)\le\left\lceil \frac{n+1}{2}\right\rceil $.\end{cor}
\begin{proof}
Let $H$ be a spanning tree of $G$. Then $\rho_{\text{opt}}(G)\le\rho_{\text{opt}}(H)\le\left\lceil \frac{n+1}{2}\right\rceil $
since additional edges cannot increase the optimal rubbling number
of a graph.
\end{proof}
Since $\rho_{\text{opt}}(K_{n})=2$ for $n\ge2$, there is no useful
lower bound for the optimal rubbling number in terms of the number
of vertices. We can still find bounds in terms of the diameter. Similar
results are presented in \cite{Optimal} for the optimal pebbling
number. 
\begin{prop}
If $G$ is a connected graph with diameter $d$, then $\left\lceil \frac{d+2}{2}\right\rceil \le\rho_{\text{opt}}(G)$.\end{prop}
\begin{proof}
Let $v_{0},v_{1},v_{2},\dots,v_{d}$ be a path where this is a shortest
path between $v_{0}$ and $v_{d}$. Build a Breadth First Search tree
from $v_{0}$. This defines a partition of the vertex set of $G$
into levels: denote the set of vertices at distance $i$ from $v_{0}$
by $L_{i}$. It is clear that $v_{i}\in L_{i}$ holds for $i\in\{0,\dots,d\}$.
Let $P_{d+1}$ denote the graph that is a path of length $d$, and
let $V(P_{d+1})=\{u_{0},\dots,u_{d}\}$. Define a mapping $\phi\colon V(G)\to V(P_{d+1})$
such that $\phi(w)=u_{i}$ for all $w\in L_{i}$. 

We claim that if a vertex $x$ in $G$ is reachable from a pebble
distribution $p$ using a sequence $s$ of rubbling moves, then $\phi(x)$
is reachable in $P_{d+1}$ from the pebble distribution that has $\sum_{w\in L_{i}}p(w)$
pebbles on $u_{i}$.

Suppose that $(a,b\a c)$ is a rubbling move in $s$. That means that
$\{a,c\}$and $\{b,c\}$ are both edges of $G.$ The properties of
the BFS tree implies that $a$ and $c$ (also $b$ and $c$) are either
in the same level or in two neighboring levels. If $a$ and $c$ (or
$b$ and $c)$ are in the same level, then we can delete this rubbling
move from the sequence, since $\phi(a)=\phi(c)$ (or $\phi(b)=\phi(c)$),
so the pebble which is moved to $c$ by $(a,b\a c)$ is already on
$\phi(c)$ in the corresponding rubbling sequence in $P_{d+1}$. 

If $c$ is not in the same level with $a$ and with $b$, then use
the rubbling move $(\phi(a),\phi(b)\a\phi(c))$. It is easy to show
by induction that this new rubbling sequence will move a pebble to
$\phi(x).$

The result now follows from this, since $\left\lceil \frac{d+2}{2}\right\rceil =\rho_{\text{opt}}(P_{d+1})\le\rho_{\text{opt}}(G)$.
\end{proof}
If the diameter of $G$ is $d$, then every vertex is reachable from
the distribution that has $2^{d}$ pebbles on a single vertex. Hence
$\rho_{\text{opt}}(G)\le2^{d}$. The following example shows that
this inequality is sharp. 

\begin{figure}
\input{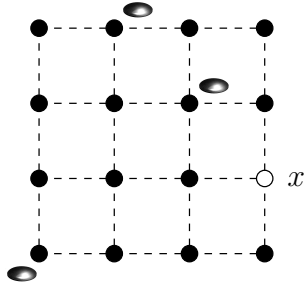}

\caption{\label{fig:Gd}Schematic representation of a graph with diameter $2$.
The dashed lines indicate the fact that any two vertices on a dashed
line are connected by an edge. The goal vertex $x$ is not reachable
from the pebble distribution shown on the figure.}

\end{figure}

\begin{prop}
For all nonnegative integer $d$ there is a graph $G$ with diameter
$d$ and $\rho_{\text{opt}}(G)=2^{d}$.\end{prop}
\begin{proof}
Let the vertex set of $G$ be the set of points in $\mathbb{N}^{d}$
with coordinates in $\{1,\ldots,2^{d}\}$. Let two vertices be connected
by an edge if they share all but one of their coordinates. It is clear
that the diameter of $G$ is $d$. Moreover, any two vertices that
has no common coordinate has distance $d$.

Consider a pebble distribution $p$ containing $2^{d}-1$ pebbles.
Since there are more possible values at the $i$-th coordinate than
the number of pebbles, there must be an $x_{i}\in\{1,\ldots,2^{d}\}$
for all $i\in\{1,\ldots,d\}$ such that no pebble has $x_{i}$ as
its $i$-th coordinate. Hence the vertex $x=\{x_{1},x_{2},\dots,x_{d}\}$
does not share any coordinate with any of the vertices containing
pebbles. In this way every pebble is at distance $d$ from $x$. Figure~\ref{fig:Gd}
shows one such pebble distribution on $G$ with vertex $x$ for $d=2$.

Assign a weight $w_{x}(p)$ to the pebble distribution $p$ in the
following way: Each pebble in $p$ gets a $1/2^{i}$ weight if its
distance from $x$ is $i$. Then $w_{x}(p)$ is the sum of the weights
for all pebbles in the distribution. Since the initial distribution
has $2^{d}-1$ pebbles, each at distance $d$ from $x$, its weight
is less than $1$. It is easy to see, that a rubbling move cannot
increase the weight. If it removes two pebbles that are at distance
$i+1$ and puts a pebble to a vertex which is one closer to $x$,
at distance $i$, then the total weight remains the same, but in all
other cases it decreases. This shows that any sequence of rubbling
moves keeps the weight smaller than $1$. However, if $x$ has a pebble,
then the weight is clearly at least $1$. Therefore $x$ is not reachable,
so $p$ is not solvable.
\end{proof}
Note that our example also serves as an example for a diameter $d$
graph with maximum pebbling number since $2^{d}=\rho_{\text{opt}}(G)\le\pi_{\text{opt}}(G)\le2^{d}$.
In fact, our example is essentially a larger version of the example
presented in \cite[Theorem 2.3]{Optimal}. The larger size allows
for a simple proof that also fills a gap in the proof of the original
example.

\bibliographystyle{amsplain}
\bibliography{pebble}

\end{document}